\documentclass[11pt, a4paper]{amsart}
\usepackage{fullpage}
\usepackage{amssymb}
\usepackage{commath}
\usepackage{empheq}
\usepackage{amsmath, amsthm}
\usepackage[foot]{amsaddr}
\usepackage{multicol}

\usepackage{tikz-cd}
\usepackage{bbm}
\usepackage{url}

\newcommand{\p}{\partial}
\newcommand{\mr}{\mathrm}
\newcommand{\ol}{\overline}
\newcommand{\ul}{\underline}
\newcommand{\DSP}{\Delta S_{+}}
\renewcommand{\phi}{\varphi}
\newcommand{\HC}{\Delta H}
\newcommand{\DHP}{\Delta H_{+}}
\newcommand{\HB}{\mathsf{H}_A}
\newcommand{\HBP}{\mathsf{H}_{A+}}
\newcommand{\HBM}{\mathsf{H}_M}
\newcommand{\llb}{\left\lbrace}
\newcommand{\rrb}{\right\rbrace}
\newcommand{\llangle}{\left\langle}
\newcommand{\rrangle}{\right\rangle}
\newcommand{\RNum}[1]{\MakeUppercase{\romannumeral #1}}
\newcommand{\CM}{\mathbf{CMod}}
\newcommand{\MC}{\mathbf{ModC}}
\newcommand{\kmod}{\mathbf{Mod}_k}
\newcommand{\C}{\mathbf{C}}

\newtheorem{thm}{Theorem}[section]

\theoremstyle{plain}
\newtheorem{prop}[thm]{Proposition}
\newtheorem{lem}[thm]{Lemma}
\newtheorem{cor}[thm]{Corollary}

\theoremstyle{definition}
\newtheorem{defn}[thm]{Definition}

\theoremstyle{remark}
\newtheorem{rem}[thm]{Remark}
\newtheorem{eg}[thm]{Example}

\makeatletter
\@namedef{subjclassname@1991}{\textup{2020} Mathematics Subject Classification}
\makeatother

\title{E-infinity structure in hyperoctahedral homology}
\author{Daniel Graves}
\address{School of Mathematics, University of Leeds, Woodhouse, Leeds, LS2 9JT, UK}
\date{}

\begin{document}

\keywords{hyperoctahedral homology, crossed simplicial group, $E_{\infty}$-algebra, Dyer-Lashof operations, Pontryagin product}
\subjclass{55N35, 13D03, 18N70, 55S12, 55N45}

\maketitle

\begin{abstract}
Hyperoctahedral homology for involutive algebras is the homology theory associated to the hyperoctahedral crossed simplicial group. It is related to equivariant stable homotopy theory via the homology of equivariant infinite loop spaces. In this paper we show that there is an E-infinity algebra structure on the simplicial module that computes hyperoctahedral homology. We deduce that hyperoctahedral homology admits Dyer-Lashof homology operations. Furthermore, there is a Pontryagin product which gives hyperoctahedral homology the structure of an associative, graded-commutative algebra. We also give an explicit description of hyperoctahedral homology in degree zero. Combining this description and the Pontryagin product we show that hyperoctahedral homology fails to preserve Morita equivalence. 
\end{abstract}

\section*{Introduction}
\emph{Hyperoctahedral homology for involutive algebras} is the homology theory associated to the hyperoctahedral crossed simplicial group \cite[Section 3]{FL}. It was introduced by Fiedorowicz \cite[Section 2]{Fie} and developed by the author \cite{GravesHHA}. The hyperoctahedral crossed simplicial group is the largest of the \emph{fundamental} crossed simplicial groups, meaning that every other crossed simplicial group is either a subobject or an extension of a subobject \cite[Theorem 3.6]{FL}. It is therefore of interest to study the properties of the associated homology theory. It has applications in equivariant stable homotopy theory. In particular, the author showed \cite[Theorem 8.8]{GravesHHA} that for a group of odd order $G$ and a commutative ground ring $k$, there is an isomorphism of graded $k$-modules $HO_{\star}(k[G])\cong H_{\star}\left(\Phi \Omega_{C_2} Q_{C_2} BG\right)$ where $BG$ is the classifying space of the group, $Q_{C_2}$ is the $C_2$-equivariant free infinite loop space functor, $\Omega_{C_2}$ is the $C_2$-equivariant based loops functor and $\Phi$ is the $C_2$-fixed point functor.

In this paper we show that the simplicial $k$-module used to compute the hyperoctahedral homology of an involutive algebra has the structure of an \emph{$E_{\infty}$-algebra} in the category of simplicial $k$-modules.   

The notion of $E_{\infty}$-algebra first appeared in work of Boardman and Vogt \cite{BV} and was developed by May \cite{MayGILS} amongst others. An $E_{\infty}$-algebra structure on an object means that is has a multiplication operation that is both associative and commutative up to all higher homotopies. One important implication of an $E_{\infty}$-algebra structure is that it yields tools for calculation. We will show that the $E_{\infty}$-algebra structure on the simplicial $k$-module which calculates hyperoctahedral homology in this paper gives rise to two such tools: \emph{Dyer-Lashof homology operations} and a \emph{Pontryagin product}. 

Dyer-Lashof homology operations are an important tool for calculation in algebraic topology and homological algebra. They were introduced in \cite{DL} and used to calculate the homology of infinite loop spaces. Consider the isomorphism of graded $k$-modules $HO_{\star}(k[G])\cong H_{\star}\left(\Phi \Omega_{C_2} Q_{C_2} BG\right)$. When we take $k[G]=\mathbb{F}_p[\mathbb{Z}]$, the right hand side has an action of the Dyer-Lashof algebra. The result of this paper shows that the left hand side also has an action of the Dyer-Lashof algebra in this case. The author hopes that this is the first stage in showing that this isomorphism of graded $k$-modules is in fact more structured. In particular, it is hoped that it preserves the Dyer-Lashof actions. This is a non-trivial statement and should certainly be investigated in the case of symmetric homology first.

The notion of a Pontryagin product has its roots in \cite{Pontryagin}. In its simplest formulation it provides a graded product structure on the singular homology of topological spaces although there are many variants and generalizations. For example, there is a Pontryagin product on $H$-spaces \cite[3.C]{Hatcher}. The Pontryagin product in this paper gives hyperoctahedral homology the structure of an associative, graded-commutative algebra. This is a remarkable property that hyperoctahedral homology shares with symmetric homology: there is a graded-commutative structure on hyperoctahedral homology, even when the algebra under consideration is not commutative! In light of this it is natural to ask whether the homology theory associated to an arbitrary crossed simplicial group has the structure of a graded ring or a graded-commutative ring. However, in general the homology theory associated to a crossed simplicial group is not known to have any structure beyond that of a graded module.

The method of proof in this paper closely follows \cite{Ault-HO} in which an $E_{\infty}$-algebra structure is found for symmetric homology. Ault's proof has four main stages:
\begin{itemize}
\item construct an analogue of the \emph{Barratt-Eccles operad} in the category of small categories;
\item construct a left operad module, in the sense of \cite[2.1.6]{Fresse}, over this operad using under-categories of the symmetric category;
\item use this left operad module structure to deduce an $E_{\infty}$-algebra structure on an object closely related to the concept of a \emph{Schur functor} \cite[Definition 1.24]{MSS};
\item deduce that this induces an $E_{\infty}$-algebra structure on the simplicial $k$-module that computes symmetric homology. 
\end{itemize} 
We will see that hyperoctahedral homology shares many of the properties of symmetric homology. The keys to finding an $E_{\infty}$-algebra structure for hyperoctahedral homology are as follows. Firstly we define the \emph{tuple category} in Section \ref{tuple-section}. In some sense this can be seen as an ``involutive" version of the material in \cite[Section 2.2]{Ault-HO}. In Section \ref{op-mod-sec} we define a left operad module over the categorical version of the Barratt-Eccles operad using under-categories of the hyperoctahedral category. The main work in this section is identifying the correct functors required to verify the left operad module structure. These are found in Lemma \ref{iso-cat-lem} and Definition \ref{functors-defn}. Once these have been established, the technical details of the proofs of Lemma \ref{iso-cat-lem}, Lemma \ref{module-lem} and Theorem \ref{main-theorem} pass over from the symmetric case \emph{mutatis mutandis}.

The paper is organized as follows. In Section \ref{background-sec} we collate the background material that we will require for the remainder of the paper. In Subsection \ref{cat-subsec} we define the categories that we will use throughout the paper. In Subsection \ref{func-hom-subsec} we recall constructions from \emph{functor homology}, in particular the tensor product of functors, $\mathrm{Tor}$ functors over a small category and simplicial $k$-modules used to compute them. Subsection \ref{HO-subsec} recalls the definition of \emph{hyperoctahedral homology} for an involutive algebra, in terms of functor homology. We recall the \emph{hyperoctahedral category} and the \emph{hyperoctahedral bar construction}. Finally, in Subsection \ref{operad-subsec} we recall the necessary material from the theory of operads. We recall the notion of an \emph{$E_{\infty}$-operad in a symmetric monoidal model category}. As examples we recall analogues of the \emph{Barratt-Eccles operad} in the category of small categories and in the category of simplicial $k$-modules. We recall the definition of a \emph{left operad module} in the sense of \cite[2.1.6]{Fresse}. 

In Section \ref{tuple-section} we define the \emph{tuple category}. For an involutive monoid $M$, the tuple category is constructed from tuples of elements in $M$, morphisms in the hyperoctahedral category and the hyperoctahedral bar construction. We show that the tuple category is symmetric strict monoidal, giving it the structure of an $E_{\infty}$-algebra in the category of small categories.

In Section \ref{op-mod-sec} we construct a left operad module over the Barratt-Eccles operad in the category of small categories, using under-categories in the hyperoctahedral category. In Section \ref{op-alg-sec} we will use the left operad module structure to deduce that there is an $E_{\infty}$-algebra structure on the simplicial $k$-module which computes hyperoctahedral homology. As a corollary we obtain Dyer-Lashof homology operations and a Pontryagin product for hyperoctahedral homology.

In Section \ref{deg-zero-sec}, we give an explicit description of hyperoctahedral homology in degree zero. We prove that for an involutive $k$-algebra $A$, the degree zero hyperoctahedral homology is isomorphic to the quotient of $A$ by the ideal generated by the relation $a_0a_1a_2-a_2\ol{a_1}a_0$. In particular degree zero hyperoctahedral homology has a ring structure induced from the one in $A$. We show that this description of degree zero hyperoctahedral homology together with the Pontryagin product implies that hyperoctahedral homology does not preserve Morita equivalence.

\subsection*{Acknowledgements}
I am very grateful to Sarah Whitehouse for her helpful comments and suggestions. I would like to thank the anonymous referee for their suggestions and for their very interesting comments and questions.

\subsection*{Conventions}
Throughout the paper, $k$ will denote a unital commutative ground ring. All the algebras considered are unital.

\section{Background material}
\label{background-sec}
We begin this section by collating definitions and notation for the categories that we will use throughout the rest of the paper.  In Subsection \ref{func-hom-subsec} we recall the constructions from functor homology necessary to define hyperoctahedral homology in Subsection \ref{HO-subsec}. We recall the hyperoctahedral category and the hyperoctahedral bar construction. We recall the simplicial $k$-module used to compute hyperoctahedral homology, which we will later show bears the structure of an $E_{\infty}$-algebra. Finally in Subsection \ref{operad-subsec} we recall the required material on operad theory, specifically the notions of $E_{\infty}$-operad in a symmetric monoidal model category and the notion of a left operad module.  

\subsection{Categories}
\label{cat-subsec}
We begin by defining the categories that we will require throughout the paper.
\begin{defn}
The following categories will be used throughout the paper.
\begin{itemize}
\item Let $\mathbf{Cat}$ be the category of small categories and functors.
\item Let $\Delta$ be the category whose objects are the sets $[n]=\llb 0,\dotsc , n\rrb$ for $n\geqslant 0$ with order-preserving maps as morphisms \cite[B.1]{Lod}.
\item Let $\mathbf{sSet}$ denote the category of simplicial sets.
\item Let $\mathbf{Mod}_k$ and $\mathbf{sMod}_k$ denote the categories of $k$-modules and simplicial $k$-modules respectively.
\item Let $\mathbf{Top}$ denote the category of compactly-generated Hausdorff topological spaces and continuous maps.
\item For a finite group $G$, let $\mathbb{E}G$ denote the category whose objects are the elements $g\in G$ with a unique morphism $g_2g_1^{-1}\in \mathrm{Hom}_{\mathbb{E}G}\left(g_1,g_2\right)$. 
\item Let $\mathbf{IMon}$ denote the category of involutive monoids and involution-preserving morphisms.
\item Let $\mathbf{S}$ denote the symmetric groupoid. The objects are the sets $\ul{n}=\llb 1,\dotsc , n\rrb$ for $n\geqslant 0$ and $\mathrm{Hom}_{\mathbf{S}}\left(\ul{n},\ul{n}\right)=\Sigma_n$.
\item For a category $\mathbf{C}$ and an object $C\in \mathbf{C}$, we denote the under-category by $C\downarrow \mathbf{C}$, see \cite[\RNum{2}.6]{CWM} for instance.
\item Let $\llb \mathbf{C}_i\rrb$ be a set of small categories whose object sets are pairwise disjoint. Let $\coprod_i \mathbf{C}_i$ denote the category whose object set is $\coprod_i \mathrm{Ob}\left(\mathbf{C}_i\right)$ and whose set of morphisms is $\coprod_i \mathrm{Mor}\left(\mathbf{C}_i\right)$. 
\end{itemize} 
\end{defn} 

\subsection{Functor homology}
\label{func-hom-subsec}
Throughout the paper we will make use of concepts from \emph{functor homology}. We recall the necessary constructions from \cite{Pir02} and \cite{GZ}.

\begin{defn}
Let $\C$ be a small category. We define the \emph{category of left $\mathbf{C}$-modules}, denoted $\CM$, to be the functor category $\mathrm{Fun}\left(\mathbf{C},\kmod\right)$. We define the \emph{category of right $\mathbf{C}$-modules}, denoted $\MC$, to be the functor category $\mathrm{Fun}\left(\mathbf{C}^{op},\kmod\right)$.
\end{defn}

It is well-known that the categories $\CM$ and $\MC$ are abelian with enough projectives and injectives, see for example \cite[Section 1.6]{Pir02}.

\begin{defn}
\label{triv-c-mod}
Let $k^{\star}$ denote the right $\C$-module that is constant at the trivial $k$-module. We will refer to this functor as the \emph{$k$-constant right $\C$-module}.
\end{defn}

\begin{defn}
\label{tensor-obj}
Let $G$ be an object of $\MC$ and $F$ be an object of $\CM$. We define the tensor product $G\otimes_{\C} F$ to be the $k$-module
\[\frac{\bigoplus_{C\in \mathrm{Ob}(\C)} G(C)\otimes_k F(C)}{\llangle G(\alpha)(x)\otimes y - x\otimes F(\alpha)(y)\rrangle}\]
where $\llangle G(\alpha)(x)\otimes y - x\otimes F(\alpha)(y)\rrangle$ is the $k$-submodule generated by the set
\[\llb G(\alpha)(x)\otimes y - x\otimes F(\alpha)(y): \alpha \in \mathrm{Hom}(\C),\,\, x\in src(G(\alpha)),\,\, y\in src(F(\alpha))\rrb.\]
\end{defn}

This quotient module is spanned $k$-linearly by equivalence classes of elementary tensors in $\bigoplus_{C\in \mathrm{Ob}(\C)} G(C)\otimes_k F(C)$ which we will denote by $\left[x\otimes y\right]$. 
\begin{defn}
One constructs the tensor product of $\C$-modules as a bifunctor 
\[-\otimes_{\mathbf{C}}-\colon \MC \times \CM\rightarrow \kmod\]
on objects by $(G,F)\mapsto G\otimes_{\C} F$. Given two natural transformations $\Theta \in \mathrm{Hom}_{\MC}\left(G, G_1\right)$ and $\Psi\in \mathrm{Hom}_{\CM}\left(F,F_1\right)$, the morphism $\Theta \otimes_{\C} \Psi$ is determined by $\left[x\otimes y\right] \mapsto \left[ \Theta_C(x)\otimes \Psi_C(y)\right]$.
\end{defn}

It is well-known that the bifunctor $-\otimes_{\C} -$ is right exact with respect to both variables and preserves direct sums, see for example \cite[Section 1.6]{Pir02}. 

\begin{defn}
We denote the left derived functors of $-\otimes_{\mathbf{C}}-$ by $\mathrm{Tor}_{\star}^{\mathbf{C}}(-,-)$.
\end{defn}

Recall the nerve of a small category $\C$ \cite[B.12]{Lod}. $N_{\star}\C$ is the simplicial set such that $N_n\C$ for $n\geqslant 1$ consists of all strings of composable morphisms of length $n$ in $\C$ and $N_0\C$ is the set of objects in $\C$. The face maps are defined to either compose adjacent morphisms in the string or truncate the string and the degeneracy maps insert identity morphisms into the string. We will denote an element of $N_n\C$ by $\left(f_n , \dotsc , f_1\right)$ where $f_i\in \mathrm{Hom}_{\C}\left(C_{i-1}, C_i\right)$.

We will be particularly interested in the case where $G$ is a functor of the form $k\left[N_{n}\left(-\downarrow \mathbf{C}\right)\right]$. That is, we take the $n^{th}$ level of the nerve of an under-category and then take the free $k$-module on this set. We can view this construction as a resolution of the $k$-constant right $\mathbf{C}$-module $k^{\star}$ in the category $\mathbf{ModC}$. We observe that in this case $k\left[N_{\star}\left(-\downarrow \mathbf{C}\right)\right]\otimes_{\mathbf{C}} F$ has the structure of a simplicial $k$-module, where the face and degeneracy maps are induced from nerve construction, and the homology of this simplicial $k$-module is $\mathrm{Tor}_{\star}^{\mathbf{C}}\left(k^{\star}, F\right)$.

There is an isomorphic construction due to Gabriel and Zisman \cite[Appendix 2]{GZ}.

\begin{defn}
\label{GabZisCpx}
Let $F\in \CM$. We define
\[C_n(\C,F)=\bigoplus_{(f_n,\dotsc ,f_1)} F(C_0)\]
where the sum runs through all elements $(f_n,\dotsc ,f_1)$ of $N_n\C$ and $f_i\in \mr{Hom}_{\C}\left(C_{i-1}, C_i\right)$. We write a generator of $C_n(\C,F)$ in the form $(f_n,\dotsc, f_1,x)$ where $(f_n,\dotsc ,f_1)\in N_n\C$ indexes the summand and $x\in F(C_0)$. 
The face maps $\p_i\colon C_n(\C,F)\rightarrow C_{n-1}(\C,F)$ are determined by
\[
\p_i(f_n,\dotsc, f_1,x)=\begin{cases}
(f_n,\dotsc ,f_2, F(f_1)(x)) & i=0,\\
(f_n, \dotsc , f_{i+1}\circ f_i,\dotsc ,f_1,x) & 1\leqslant i \leqslant n-1,\\
(f_{n-1},\dotsc ,f_1, x) & i=n.
\end{cases}\]
The degeneracy maps insert identity maps into the string. We will denote the homology of the associated chain complex by $H_{\star}\left(\mathbf{C},F\right)$.
\end{defn}

\begin{rem}
As mentioned above, for a left $\mathbf{C}$-module $F$ there is an isomorphism of simplicial $k$-modules $C_{\star}\left(\mathbf{C},F\right)\cong k\left[N_{\star}\left(-\downarrow \mathbf{C}\right)\right]\otimes_{\mathbf{C}} F$. This is well-known; for example, details of this isomorphism are discussed in \cite[Section 1.3]{Ault} for specific choices of $\mathbf{C}$ and $F$.
\end{rem}

\subsection{Hyperoctahedral homology}
\label{HO-subsec}
Recall from \cite[Definition 1.1]{FL} that a family of groups $\llb G_n\rrb_{n\geqslant 0}$ is a \emph{crossed simplicial group} if there exists a small category $\Delta G$ such that: the objects of $\Delta G$ are the sets $[n]$ for $n\geqslant 0$; $\Delta G$ contains $\Delta$ as a subcategory; $\mathrm{Aut}_{\Delta G}([n])=G_n$; and any morphism in $\mathrm{Hom}_{\Delta G}\left([n] , [m]\right)$ can be written uniquely as a pair $\left(\phi , g\right)$ where $g\in G_n$ and $\phi \in \mathrm{Hom}_{\Delta}\left([n],[m]\right)$.

The final condition ensures that the category $\Delta G$ has a well defined composition, which is denoted by $\left(\psi, h\right)\circ\left(\phi, g\right)=\left(\psi\circ h_{\star}(\phi),\phi^{\star}(h)\circ g\right)$ where the morphisms $h_{\star}(\phi)$ and $\phi^{\star}(h)$ are determined by the crossed simplicial group $\llb G_n\rrb$.

\begin{defn}
For $n\geqslant 0$, the \emph{hyperoctahedral group} $H_{n+1}$ is defined to be the semi-direct product $C_2^{n+1}\rtimes \Sigma_{n+1}$, where $C_2=\llangle t\mid t^2=1\rrangle$ and $\Sigma_{n+1}$ acts on $C_2^{n+1}$ by permuting the factors.
\end{defn}

It is shown in \cite[3.3, 3.4]{FL} that the family of hyperoctahedral groups $\llb H_{n+1}\rrb_{n \geqslant 0}$ and the family of symmetric groups $\llb \Sigma_{n+1}\rrb_{n\geqslant 0}$ are crossed simplicial groups.

\begin{defn}
Let $\Delta S$ and $\Delta H$ denote the categories associated to the symmetric and hyperoctahedral crossed simplicial groups respectively. The composition laws in each category are determined by the relations given in \cite[E.6.1.7]{Lod} and \cite[Definition 1.2]{GravesHHA} respectively.
\end{defn}

\begin{rem}
Let $\Delta S_{+}$ and $\Delta H_{+}$ denote the categories formed from $\Delta S$ and $\Delta H$ by appending an initial object, which will be denoted by $[-1]$. Both $\Delta S_{+}$ and $\Delta H_{+}$ are symmetric strict monoidal categories, as proved in \cite[Proposition 9]{Ault} and \cite[Proposition 6.2]{GravesHHA} respectively.
\end{rem} 

\begin{rem}
Fiedorowicz and Loday provide a classification theorem for crossed simplicial groups \cite[Theorem 3.6]{FL}. This shows that there are seven \emph{fundamental crossed simplicial groups}. The largest of these is the hyperoctahedral crossed simplicial group; the remainder being subgroups of hyperoctahedral groups. Any other crossed simplicial group is an extension of a fundamental crossed simplicial group. From this perspective it is important to understand the properties of hyperoctahedral homology. However it also raises interesting questions. Fiedorowicz and Loday prove that any semi-direct extension of the symmetric groups is a crossed simplical group \cite[Theorem 3.10]{FL}. The author plans to investigate the homology theories associated to these crossed simplicial groups.
\end{rem}

\begin{defn}
Let $A$ be an involutive, associative $k$-algebra with the involution denoted by $a \mapsto \ol{a}$. The \emph{hyperoctahedral bar construction} is the functor $\HB\colon \Delta H\rightarrow \kmod$
given on objects by $[n]\mapsto A^{\otimes n+1}$ for $n\geqslant 0$. Let $(\phi ,g )\in \mr{Hom}_{\HC}\left([n],[m]\right)$. Then $\phi \in \mr{Hom}_{\Delta}\left([n] , [m]\right)$ and $g=\left( z_0,\dotsc , z_n; \sigma\right) \in H_{n+1}$. We define
\[\HB\left(\phi ,g\right)(a_0\otimes \cdots \otimes a_n)=\left( \underset{_{i\in (\phi\circ \sigma)^{-1}(0)}}{{\prod}} a_i^{z_i}\right)\otimes \cdots \otimes \left(\underset{_{i\in (\phi\circ \sigma)^{-1}(m)}}{{\prod}} a_i^{z_i} \right)\]
on elementary tensors and extend $k$-linearly, where the product is ordered according to the map $\phi$ and
\[a^{z_i}=
\begin{cases}
a & z_i=1\\
\ol{a} & z_i=t.
\end{cases}\]
Note that an empty product is defined to be the multiplicative unit $1_A$.
\end{defn}

\begin{defn}
Let $A$ be an involutive, associative algebra. For $n\geqslant 0$, we define the \emph{$n^{th}$ hyperoctahedral homology of $A$} to be
\[HO_{n}(A)\coloneqq\mr{Tor}_{n}^{\HC}\left(k^{\star},\mathsf{H}_A\right).\]
\end{defn}

\begin{rem} 
We recall from \cite[Definition 6.3]{GravesHHA} that the hyperoctahedral bar construction extends to a functor $\mathsf{H}_{A+}\in \mathrm{Fun}\left(\DHP , \mathbf{Mod}_k\right)$ by defining $\mathsf{H}_{A+}([-1])=k$ and, for $n\geqslant 0$ and $i_n\in \mathrm{Hom}_{\DHP}\left([-1],[n]\right)$, defining $\mathsf{H}_{A+}(i_n)$ to be the inclusion of $k$-algebras $k\rightarrow A^{\otimes (n+1)}$. Furthermore, by \cite[Theorem 6.4]{GravesHHA}, there is an isomorphism of graded $k$-modules $HO_{\star}(A)\cong H_{\star}\left(\DHP , \mathsf{H}_{A+}\right)$.
\end{rem}

\begin{defn}
Let $M$ be a monoid with involution. We define a functor $\HBM\colon \DHP\rightarrow \mathbf{IMon}$ on objects by
\[\HBM\left([n]\right)=\begin{cases}
M^{n+1} & n\geqslant 0\\
\emptyset & n=-1.
\end{cases}\]
where $M^{n+1}$ denotes the $(n+1)$-fold Cartesian product. $\HBM$ is defined on morphisms in $\HC$ analogously to $\HB$. We define $\HBM\left(i_n\right)$ to be the unique map $\emptyset \rightarrow M^{n+1}$. We call $\HBM$ the \emph{hyperoctahedral bar construction for monoids}.
\end{defn}

\subsection{Operad theory}
\label{operad-subsec}

Recall \cite[4.2.6]{Hovey} that a symmetric monoidal model category is category that bears both the structure of a symmetric monoidal category \cite[4.1.4]{Hovey} and of a model category \cite[1.1.3]{Hovey} subject to some compatibility conditions. 

\begin{eg}
We will be particularly interested in the following symmetric monoidal model categories.
\begin{itemize}
\item The category $\mathbf{sSet}$ with the level-wise Cartesian product and the Quillen model structure \cite[4.2.8]{Hovey}.
\item The category $\mathbf{sMod}_k$ with the level-wise tensor product of $k$-modules and the projective model structure \cite[\RNum{2}.4.12]{Quillen}. 
\item The category $\mathbf{Cat}$ with the product of categories and the Thomason model structure \cite{Thomason}.
\end{itemize}
\end{eg}   

\begin{defn}
Let $\mathbf{C}$ be a small symmetric monoidal model category with unit $\mathbbm{1}$. An \emph{$E_{\infty}$-operad} in $\mathbf{C}$ is a functor $\mathcal{O}\in\mathrm{Fun}\left(\mathbf{S}^{op}, \mathbf{C}\right)$ such that
\begin{itemize}
\item the conditions of an operad \cite[Definition 1.1]{MayGILS} are satisfied;
\item for each $n\geqslant 0$, $\mathcal{O}(n)$ is weakly equivalent to $\mathbbm{1}$ in the model structure;
\item for each $n\geqslant 0$, the action of $\Sigma_n$ on $\mathcal{O}(n)$ is free.
\end{itemize}
\end{defn}

\begin{eg}
Some interesting examples of $E_{\infty}$-operads in $\mathbf{Top}$ are the \emph{little $\infty$-cubes operad} \cite[Section 4]{MayGILS}, the \emph{Barratt-Eccles operad} \cite{BE}, \cite[1.1]{BeFr} and the \emph{hyperoctahedral operad} \cite[Definition 7.9]{GravesHHA}.
\end{eg}

\begin{rem}
Observe that if $\mathcal{O}$ is an $E_{\infty}$-operad in a symmetric monoidal model category $\mathbf{C}$ and $F\in \mathrm{Fun}\left(\mathbf{C},\mathbf{D}\right)$ is symmetric monoidal then there is an $E_{\infty}$-operad in $\mathbf{D}$ induced from $\mathcal{O}$ by $F$. We will be particularly interested in the categories $\mathbf{Cat}$, $\mathbf{sSet}$ and $\mathbf{sMod}_k$ with the nerve functor $N\in \mathrm{Fun}\left(\mathbf{Cat},\mathbf{sSet}\right)$ and the free $k$-module functor $k[-]\in \mathrm{Fun}\left(\mathbf{sSet}, \mathbf{sMod}_k\right)$.
\end{rem}

\begin{eg}
Following \cite[Example 2.7]{Ault-HO} we describe an analogue of the Barratt-Eccles operad in $\mathbf{Cat}$. The nerve functor then recovers a familiar description of the Barratt-Eccles operad in the category of simplicial sets and the free $k$-module functor induces an analogue of the Barratt-Eccles operad in the category $\mathbf{sMod}_k$. Following Ault, the operads in $\mathbf{Cat}$ and $\mathbf{sMod}_k$ will be denoted by $\mathcal{D}_{\mathbf{Cat}}$ and $\mathcal{D}_{\mathbf{Mod}}$ respectively.

For $m\geqslant 0$ let $\mathcal{D}_{\mathbf{Cat}}(m)=\mathbb{E}\Sigma_m$. Let $m,k_1,\dotsc , k_m\geqslant 0$ and let $k=\sum_{i=1}^m k_i$. The operad structure maps are the functors
\[\mathcal{D}_{\mathbf{Cat}}(m)\times \mathcal{D}_{\mathbf{Cat}}(k_1)\times \cdots \times \mathcal{D}_{\mathbf{Cat}}(k_m)\rightarrow \mathcal{D}_{\mathbf{Cat}}(k)\]
given on objects by 
\[\left(\sigma , \tau_1,\dotsc , \tau_m\right) \mapsto \tau_{\sigma^{-1}(1)}\times \cdots \times \tau_{\sigma^{-1}(m)}.\]
That is, we form the permutation where $\tau_i$ acts on a block of size $k_i$ and we permute the blocks according to the permutation $\sigma \in \Sigma_m$.

The right action of $\Sigma_m$ of $\mathbb{E}\Sigma_m$ is given by composition of group elements and is therefore a free action. The category $\mathbb{E}\Sigma_m$ is contractible since the nerve of the category $\mathbb{E}\Sigma_m$ is a simplicial model of the total space of $\Sigma_m$. Therefore $\mathcal{D}_{\mathbf{Cat}}$ is an $E_{\infty}$-operad in $\mathbf{Cat}$. 
\end{eg}

\begin{defn}
Let $\mathbf{C}$ be a symmetric monoidal model category and let $\mathcal{O}$ be an $E_{\infty}$-operad in $\mathbf{C}$. An $\mathcal{O}$-algebra \cite[Definition \RNum{2}.1.20]{MSS} is called an \emph{$E_{\infty}$-algebra in $\mathbf{C}$}.
\end{defn}

\begin{eg}
As shown in \cite[2.5]{Ault-HO}, a symmetric strict monoidal category has the structure of a $\mathcal{D}_{\mathbf{Cat}}$-algebra and is therefore an $E_{\infty}$-algebra in $\mathbf{Cat}$.
\end{eg}

\begin{defn}
Let $\left(\mathbf{C}, \otimes\right)$ be a symmetric monoidal category. Let $\mathcal{O}$ be an operad in $\mathbf{C}$. A \emph{left $\mathcal{O}$-module} is a functor $\mathcal{M}\in \mathrm{Fun}\left(\mathbf{S}^{op},\mathbf{C}\right)$ together with composition products
\[\mathcal{O}(n)\otimes \mathcal{M}(p_1)\otimes \cdots \otimes \mathcal{M}(p_n)\rightarrow \mathcal{M}\left(p_1+\cdots +p_n\right)\]
satisfying associativity, left unit and equivariance conditions analogous to those of an operad.
\end{defn}

\begin{prop}
\label{schur-prop}
Let $\left(\mathbf{C},\oplus , \otimes\right)$ be a cocomplete distributive symmetric monoidal category. Let $\mathcal{O}$ be an operad in $\mathbf{C}$. Let $\mathcal{M}$ be a left $\mathcal{O}$-module and let $Z$ be an object in $\mathbf{C}$. Then
\[\mathcal{M}\llangle Z\rrangle\coloneqq \bigoplus_{m\geqslant 0} \mathcal{M}\left(\ul{m}\right)\otimes_{\Sigma_m} Z^{\otimes m}\]
admits the structure of an $\mathcal{O}$-algebra.
\end{prop}
\begin{proof}
This is stated as \cite[Proposition 3.7]{Ault-HO}. It can be proved analogously to \cite[Proposition 1.25]{MSS} using the left operad module structure maps in place of the operad structure maps.
\end{proof}

\section{The tuple category}
\label{tuple-section}
In this section we define the \emph{tuple category}. Given an involutive monoid $M$ we construct a category whose objects are tuples of elements in $M$ and whose morphisms are constructed from the hyperoctahedral category $\DHP$. We show that the tuple category is symmetric strict monoidal, making it an $E_{\infty}$-algebra in $\mathbf{Cat}$. This structure is key in the proof of Lemma \ref{module-lem}. 

\begin{defn}
\label{tup-cat-defn}
For an involutive monoid $M$ we define the \emph{tuple category of $M$}, denoted by $T(M)$ as follows. The objects of $T(M)$ are finite, possibly empty, tuples of elements in $M$. The empty tuple is denoted by $()$. For each morphism $f\in \mathrm{Hom}_{\DHP}\left([p-1],[q-1]\right)$ for $p,\,q\geqslant 0$ and $\mathbf{m}\in M^p$ there is a morphism $\left(f,\mathbf{m}\right)\in \mathrm{Hom}_{T(M)}\left(\mathbf{m}, \HBM(f)(\mathbf{m})\right)$. Composition is defined via the composition of morphisms in the category $\DHP$ and is well-defined by the functoriality of $\HBM$.
\end{defn}

\begin{defn}
We define a functor $T(-)\in \mathrm{Fun}\left(\mathbf{IMon},\mathbf{Cat}\right)$ on objects by the assignment of Definition \ref{tup-cat-defn}. Let $\mathbf{m}\in M^p$ and let $\left(f,\mathbf{m}\right)\in \mathrm{Hom}_{T(M)}\left(\mathbf{m} , \mathsf{H}_M(f)(\mathbf{m})\right)$. For a morphism $\phi \in \mathrm{Hom}_{\mathbf{IMon}}\left(M,N\right)$, $T(f)\in \mathrm{Hom}_{\mathbf{Cat}}\left(T(M),T(N)\right)$ is the functor defined on objects by $T(\phi)(\mathbf{m})=\phi(\mathbf{m})$ and on morphisms by $T(\phi)(f,\mathbf{m})=(f,\phi(\mathbf{m}))$, where $\phi$ is evaluated on the tuple $\mathbf{m}$ point-wise.
\end{defn}

\begin{lem}
\label{tuple-cat-lem}
Let $M$ be an involutive monoid. The tuple category $T(M)$ has the structure of a $\mathcal{D}_{\mathbf{Cat}}$-algebra. That is, it is an $E_{\infty}$-algebra in $\mathbf{Cat}$.
\end{lem}
\begin{proof}
The product of two tuples in $T(M)$ is given by concatenation. The fact this product is symmetric strict monoidal follows from the fact $\DHP$ is symmetric strict monoidal under the disjoint union.
\end{proof}

\section{A left operad module}
\label{op-mod-sec}
In this section we will define a left operad module over $\mathcal{D}_{\mathbf{Cat}}$ in terms of under-categories of the hyperoctahedral groups. We will use the left operad module defined in this section, in conjunction with Proposition \ref{schur-prop}, to give an $E_{\infty}$-algebra structure on $C_{\star}\left(\DHP , \HBP\right)$, the simplicial $k$-module which computes hyperoctahedral homology.

\begin{defn}
\label{mod-func-defn}
We define a functor $\mathcal{I}\mathcal{F}(-)\in \mathrm{Fun}\left(\mathbf{S}^{op},\mathbf{Cat}\right)$ as follows. On objects we define $\mathcal{I}\mathcal{F}\left(\ul{m}\right)=[m-1]\downarrow \DHP$. For $\sigma\in \Sigma_m$ and $(\phi,g) \in [m-1]\downarrow \DHP$ we define $\mathcal{I}\mathcal{F}(\sigma)(\phi , g)=(\phi,g\circ \sigma)$.
\end{defn}

\begin{rem}
Henceforth, for the sake of neatness, we will omit the underline on an object when we apply the functor $\mathcal{I}\mathcal{F}$.
\end{rem} 

\begin{defn}
\label{lambda-defn}
Let $m,j_1,\dotsc , j_m\geqslant 0$ and let $j=\sum_{i=1}^m j_i$. For $1\leqslant i \leqslant m$ let $f_i\in \mathrm{Hom}_{\DHP}\left([j_i-1],[p_i-1]\right)$ and let $g_i\in \mathrm{Hom}_{\DHP}\left([p_i-1],[q_i-1]\right)$.

We define a family of functors
\[\lambda\colon \mathcal{D}_{\mathbf{Cat}}(m)\times \prod_{i=1}^m \mathcal{I}\mathcal{F}(j_i)\rightarrow \mathcal{I}\mathcal{F}(j)\]
on objects by
\[\lambda\left(\sigma, f_1,\dotsc , f_m\right) = f_{\sigma^{-1}(1)}\times \cdots \times f_{\sigma^{-1}(m)}.\]
That is, $f_i$ acts on a block of size $j_i$ and we permute the $m$ blocks according to $\sigma \in \Sigma_m$.
The functor $\lambda$ is defined on morphisms by
\[\lambda\left(\tau\sigma^{-1},g_1,\dotsc , g_m\right)\left(\sigma, f_1,\dotsc , f_m\right)=g_{\tau^{-1}(1)}f_{\tau^{-1}(1)}\times \cdots \times g_{\tau^{-1}(m)}f_{\tau^{-1}(m)}.\]
That is, we compose $g_i$ with $f_i$, set this to act on a block of size $j_i$ and then permute the blocks according to the permutation $\tau\in \Sigma_m$. The functoriality of $\lambda$ follows from the functoriality of the composition of morphisms in the under-category. 
\end{defn}

We will show in Lemma \ref{module-lem} that the structure from Definitions \ref{mod-func-defn} and \ref{lambda-defn} define the structure of a left $\mathcal{D}_{\mathbf{Cat}}$-module on $\mathcal{I}\mathcal{F}(-)$.

\begin{defn}
Let $X=\llb x_1, \ol{x_1}, x_2 , \ol{x_2},\dotsc\rrb$ be a set of formal indeterminates. Let $F(X)$ denote the free involutive monoid on $X$.
\end{defn}

\begin{rem}
Observe that the set $X$ has a $C_2$-action given by $x_i\mapsto \ol{x_i}$.
Recall that a category $\mathbf{C}$ is called \emph{discrete} if the only morphisms in $\mathbf{C}$ are identity morphisms. Note that for each $m\geqslant 0$ we can consider the Cartesian product $X^m$ as a discrete category. Furthermore, this category has a left action of the hyperoctahedral group $H_m$ given by applying the $C_2$-action and permuting the factors.
\end{rem}

\begin{rem}
Observe that for each $m\geqslant 0$ the category $\mathcal{I}\mathcal{F}(m)$ has a right action of the hyperoctahedral group $H_m$ defined by $\left(\phi , g\right)\bullet h=\left(\phi  , g\circ h\right)$.
\end{rem}

\begin{lem}
\label{iso-cat-lem}
There is an isomorphism of categories
\[E\colon \coprod_{m\geqslant 0} \left(\mathcal{I}\mathcal{F}(m) \times_{H_m} X^m\right)\rightarrow T(F(X))\]
induced from the evaluation maps $\mathcal{I}\mathcal{F}(m)\times X^m\rightarrow T(F(X))$ given by 
\[(f,y_1,\dotsc , y_m)\mapsto \mathsf{H}_M(f)(y_1,\dotsc , y_m).\]
\end{lem}
\begin{proof}
The proof is analogous to \cite[Lemma 3.4]{Ault-HO}. One shows that the evaluation maps factor through $\mathcal{I}\mathcal{F}(m)\times_{H_m} X^m$ using the functoriality of $\mathsf{H} _M$. One then constructs the inverse functor, which is induced from the maps that completely factor the monomials in a tuple in $T(F(X))$.
\end{proof}

For the remainder of this section, fix $m,j_1,\dotsc , j_m\geqslant 0$ and let $j=\sum_{i=1}^m j_i$.

\begin{defn}
Let $\mathbf{y}=\left(y_1,\dotsc , y_j\right)\in X^j$. Define $\mathbf{y}_1=\left(y_1,\dotsc ,y_{j_1}\right)\in X^{j_1}$ and, for $2\leqslant i\leqslant m$, define $\mathbf{y}_i=\left(y_{j_1+\cdots +j_{i-1}+1},\dotsc , y_{j_1+\cdots +j_i}\right)\in X^{j_i}$.
\end{defn}

\begin{defn}
\label{functors-defn}
We define the functors we require to verify that $\mathcal{I}\mathcal{F}$ is a left $\mathcal{D}_{\mathbf{Cat}}$-module.
\begin{itemize}
\item For $1\leqslant i \leqslant m$, let $I_i\colon \mathcal{I}\mathcal{F}(j_i)\rightarrow \mathcal{I}\mathcal{F}(j_i)\times_{H_{j_i}} X^{j_i}$
be the inclusion functor $f\mapsto (f, \mathbf{y}_i)$.
\item Let $\mathbf{I}=I_1\times \cdots \times I_m$. 
\item Let $I\colon \mathcal{I}\mathcal{F}(j)\rightarrow \mathcal{I}\mathcal{F}(j) \times_{H_j} X^j$ be the inclusion functor $f\mapsto (f, \mathbf{y})$.
\item For $i\geqslant 0$, let 
\[J_i\colon \mathcal{I}\mathcal{F}(i)\times_{H_i} X^i\rightarrow \coprod_{i\geqslant 0} \left(\mathcal{I}\mathcal{F}(i)\times_{H_i} X^i\right)\]
be the inclusion of categories.
\item Let $\mathbf{J}=J_{j_1}\times \cdots \times J_{j_m}$.
\end{itemize}
\end{defn}

\begin{lem}
\label{module-lem}
The functors $\lambda$ of Definition \ref{lambda-defn} define a left $\mathcal{D}_{\mathbf{Cat}}$-module structure on the functor $\mathcal{I}\mathcal{F}(-)$ of Definition \ref{mod-func-defn}.
\end{lem}
\begin{proof}
Recall from Lemma \ref{tuple-cat-lem} that the tuple category $T(F(X))$ has the structure of a $\mathcal{D}_{\mathbf{Cat}}$-algebra. Therefore, for $m\geqslant 0$, we have operad algebra structure maps
\[\theta\colon \mathcal{D}_{\mathbf{Cat}}(m)\times T(F(X))^m \rightarrow T(F(X))\]
as defined in \cite[Section 2.5]{Ault-HO}.

Recall the isomorphism of categories from Lemma \ref{iso-cat-lem} and the functors of Definition \ref{functors-defn}.

Analogously to \cite[Diagram (16)]{Ault-HO}, there is a commutative diagram
\begin{center}
\begin{tikzcd}
\mathcal{D}_{\mathbf{Cat}}(m)\times \displaystyle\prod_{i=1}^m \mathcal{I}\mathcal{F}(j_i)\arrow[rrr, "\lambda "]\arrow[d, "{id \times \mathbf{I}}",swap]&&&\mathcal{I}\mathcal{F}(j)\arrow[d, "I"]\\
\mathcal{D}_{\mathbf{Cat}}(m) \times \displaystyle\prod_{i=1}^m \left(\mathcal{I}\mathcal{F}(j_i)\times_{H_{j_i}} X^{j_i}\right)\arrow[d, "{id \times \mathbf{J}}",swap]&&& \mathcal{I}\mathcal{F}(j)\times_{H_j} X^j \arrow[d, "J_j"]\\
\mathcal{D}_{\mathbf{Cat}}(m)\times \left[\displaystyle\coprod_{i\geqslant 0} \left(\mathcal{I}\mathcal{F}(i)\times_{H_i} X^i\right)\right]^{\times m}\arrow[d, "{id\times E^{\times m}}",swap]&&& \displaystyle\coprod_{i\geqslant 0} \left(\mathcal{I}\mathcal{F}(i)\times_{H_i} X^i\right)\arrow[d, "E"]\\
\mathcal{D}_{\mathbf{Cat}}(m)\times T(F(X))^m\arrow[rrr, "\theta ",swap] &&& T(F(X))
\end{tikzcd}
\end{center}
in $\mathbf{Cat}$.

The associativity condition of the functors $\lambda$ is induced from the associativity conditions of the operad algebra structure maps $\theta$, since all the vertical functors are injective. The left unit condition is satisfied straightforwardly. The equivariance conditions are satisfied similarly to the diagrams labelled \textbf{Equivariance A} and \textbf{Equivariance B} in \cite[Section 3.3]{Ault-HO}.
\end{proof}

\begin{defn}
\label{mod-module}
Let $\widetilde{\mathcal{I}\mathcal{F}}(-)\coloneqq k\left[N_{\star}\left(-\downarrow \DHP\right)\right]\in \mathrm{Fun}\left(\mathbf{S}^{op},\mathbf{sMod}_k\right)$.
\end{defn}

\begin{cor}
There is a left $\mathcal{D}_{\mathbf{Mod}}$-module structure on $\widetilde{\mathcal{I}\mathcal{F}}(-)$.
\end{cor}
\begin{proof}
The required structure is induced from the $\mathcal{D}_{\mathbf{Cat}}$-module structure on $\mathcal{I}\mathcal{F}(-)$ by the composite of the nerve functor and the free $k$-module functor.
\end{proof}

\section{E-infinity structure}
\label{op-alg-sec}
In this section we apply Proposition \ref{schur-prop} to obtain an $E_{\infty}$-algebra in the category of simplicial $k$-modules, constructed from the left $\mathcal{D}_{\mathbf{Mod}}$-module $\widetilde{\mathcal{I}\mathcal{F}}$ of Definition \ref{mod-module} and the hyperoctahedral bar construction. We prove that the simplicial $k$-module $C_{\star}\left(\DHP , \HBP\right)$ is a quotient of this and that the quotient map induces the structure of an $E_{\infty}$-algebra structure on $C_{\star}\left(\DHP , \HBP\right)$. As a corollary we observe that hyperoctahedral homology of an involutive algebra admits Dyer-Lashof operations and a Pontryagin product.

\begin{lem}
\label{DS-lem}
Let $A$ be an involutive $k$-algebra. The simplicial $k$-module $\widetilde{\mathcal{I}\mathcal{F}}(-)\otimes_{\mathrm{Aut}\DSP} \HBP$ admits the structure of an $E_{\infty}$-algebra in $\mathbf{sMod}_k$.
\end{lem}
\begin{proof}
One identifies 
\[\widetilde{\mathcal{I}\mathcal{F}}(-)\otimes_{\mathrm{Aut}\DSP} \HBP = \bigoplus_{n\geqslant 0} \widetilde{\mathcal{I}\mathcal{F}}(n) \otimes_{\Sigma_n} A^{\otimes n}= \widetilde{\mathcal{I}\mathcal{F}}\llangle A\rrangle.\]
The result now follows from Proposition \ref{schur-prop} where we set $\mathbf{C}=\mathbf{sMod}_k$, $\mathcal{O}=\mathcal{D}_{\mathbf{Mod}}$, $\mathcal{M}=\widetilde{\mathcal{I}\mathcal{F}}(-)$ and $Z=A$.
\end{proof}

\begin{defn}
\label{Quotient-defn}
Let $q\colon \widetilde{\mathcal{I}\mathcal{F}}\llangle A\rrangle \rightarrow C_{\star}\left(\DHP , \HBP\right)$ denote the quotient map of simplicial $k$-modules induced from the inclusion of categories $\mathrm{Aut}\DSP\hookrightarrow \DSP \hookrightarrow \DHP$.
\end{defn}

\begin{defn}
\label{structure-map-defn}
We denote by
\[\nu\colon \mathcal{D}_{\mathbf{Mod}}(m) \otimes_{\Sigma_m} \widetilde{\mathcal{I}\mathcal{F}}\llangle A\rrangle^{\otimes m}\rightarrow \widetilde{\mathcal{I}\mathcal{F}}\llangle A\rrangle\]
the operad algebra structure maps implied by Lemma \ref{DS-lem}.
\end{defn}

\begin{thm}
\label{main-theorem}
Let $A$ be an involutive $k$-algebra. The simplicial $k$-module $C_{\star}\left(\DHP , \HBP\right)$, which computes the hyperoctahedral homology $HO_{\star}(A)$, bears the structure of an $E_{\infty}$-algebra in the category $\mathbf{sMod}_k$.
\end{thm}
\begin{proof}
Analogously to \cite[Lemma 3.10]{Ault-HO}, one can show that the $E_{\infty}$-algebra structure maps $\nu$ of Definition \ref{structure-map-defn} and the quotient map $q$ of Definition \ref{Quotient-defn} are compatible in the sense that the diagram
\begin{center}
\begin{tikzcd}
\mathcal{D}_{\mathbf{Mod}}(m)\otimes_{\Sigma_m} \widetilde{\mathcal{I}\mathcal{F}}\llangle A\rrangle^{\otimes m}\arrow[rr, "\nu "]\arrow[d, "{id\otimes q^{\otimes m}}", swap] && \widetilde{\mathcal{I}\mathcal{F}}\llangle A\rrangle \arrow[d, "q"]\\
\mathcal{D}_{\mathbf{Mod}}\otimes_{\Sigma_m}  C_{\star}\left(\DHP , \HBP\right)^{\otimes m}\arrow[rr, "\nu ", swap] && C_{\star}\left(\DHP , \HBP\right)
\end{tikzcd}
\end{center}
commutes. That is, the structure map $\nu$ remains well defined upon passing to the quotient and we deduce that $C_{\star}\left(\DHP , \HBP\right)$ has the structure of a $\mathcal{D}_{\mathbf{Mod}}$-algebra.
\end{proof}

\begin{cor}
\label{Main-cor}
Let $A$ be an involutive $k$-algebra.
\begin{enumerate}
\item When $k=\mathbb{F}_p$ for a prime $p$, the hyperoctahedral homology $HO_{\star}(A)$ admits Dyer-Lashof homology operations.
\item For any commutative ground ring $k$, $HO_{\star}(A)$ admits a Pontryagin product, giving it the structure of a unital, associative, graded-commutative algebra.
\end{enumerate} 
\end{cor}
\begin{proof}
The first part of the corollary follows directly from Theorem \ref{main-theorem} and \cite[Chapter \RNum{1} Theorem 1.1]{CLM}. The second part of the corollary follows by defining the product
\[C_{\star}\left(\DHP , \HBP\right)^{\otimes 2}\hookrightarrow \mathcal{D}_{\mathbf{Mod}}(2)\otimes_{\Sigma_2} C_{\star}\left(\DHP , \HBP\right)^{\otimes 2}\xrightarrow{\nu} C_{\star}\left(\DHP , \HBP\right)\]
where $\nu$ is the $\mathcal{D}_{\mathbf{Mod}}$-algebra structure map from Theorem \ref{main-theorem}. The fact that the product is unital follows from the fact that our analogues of the Barratt-Eccles operad are unital operads.
\end{proof}

\section{Hyperoctahedral homology in degree zero}
\label{deg-zero-sec}
In this section we give an explicit description of hyperoctahedral homology in degree zero. We will show that for a unital involutive $k$ algebra $A$, $HO_0(A)$ is isomorphic to the quotient of $A$ by the ideal $\left(a_0a_1a_2-a_2\ol{a_1}a_0\right) $. In particular, we deduce that $HO_0(A)$ has a ring structure.

The material in this section is best presented using an isomorphic variant of the category $\Delta H$, namely the category of involutive non-commutative sets $\mathcal{I}\mathcal{F}(as)$. The isomorphism of \cite[Theorem 1.7]{GravesHHA} tells that a morphism in $\mathrm{Hom}_{\Delta H} \left( [n],[m]\right)$ can be uniquely described as a map of sets $f\colon [n]\rightarrow [m]$  such that for each element $i\in [m]$ the preimage $f^{-1}(i)$ has a total ordering and each element comes adorned with a superscript label from the group $C_2$. We call such a set a \emph{$C_2$-labelled ordered set}. We also recall that a $C_2$-labelled ordered set $X$ has a $C_2$-action, $X^t$ that reverses the total ordering and multiplies each of the labels by $t\in C_2$.

We will require two maps in particular.

\begin{defn}
For each $n\geqslant 0$, let $\mu_n \in \mathrm{Hom}_{\mathcal{I}\mathcal{F}(as)}\left([n],[0]\right)$ be defined by \[\mu_n^{-1}(0)=\llb 0^1<1^1<\cdots <n^1\rrb.\]
\end{defn}

\begin{defn}
Let $\nu \in \mathrm{Hom}_{\mathcal{I}\mathcal{F}(as)}\left([2],[0]\right)$ be defined by $\nu^{-1}(0)=\llb 2^1<1^t<0^1\rrb$.
\end{defn}

We begin by constructing a partial resolution of the $k$-constant right $\mathcal{I}\mathcal{F}(as)$-module $k^{\star}$.

\begin{lem}
\label{ESLem}
For each $n\geqslant 0 $ there is an exact sequence of $k$-modules
\[0\leftarrow k\xleftarrow{\varepsilon} k\left[\mathrm{Hom}_{\mathcal{I}\mathcal{F}(as)}\left([n],[0]\right)\right] \xleftarrow{\rho} k\left[\mathrm{Hom}_{\mathcal{I}\mathcal{F}(as)}\left([n],[2]\right)\right]\]
where $\varepsilon$ is defined by $\varepsilon(f)=1_k$ on generators and extended $k$-linearly. The map $\rho$ is determined on generators by $\rho(g)=\mu_2\circ g - \nu \circ g$ and extended $k$-linearly.
\end{lem}
\begin{proof}
The augmentation map $\varepsilon$ is surjective by definition. The composite $\varepsilon\circ \rho$ is zero since the image of $\rho$ on generators consists of two morphisms with differing signs. Therefore $\mathrm{Im}(\rho)\subseteq \mathrm{Ker}(\varepsilon)$. It remains to show that $\mathrm{Ker}(\varepsilon)\subseteq \mathrm{Im}(\rho)$. 

We note that $\mathrm{Ker}(\varepsilon)$ is spanned by all differences $f - f^{\prime}$ for $f$, $f^{\prime}\in \mathrm{Hom}_{\mathcal{I}\mathcal{F}(as)}\left([n],[0]\right)$. In fact, it is spanned by all differences of the form $f-\mu_n$ for $f\in \mathrm{Hom}_{\mathcal{I}\mathcal{F}(as)}\left([n],[0]\right)$. Certainly all the differences of the form $f-\mu_n$ are in the kernel of $\varepsilon$ and we can write any difference $f-f^{\prime}$ as $\left(f-\mu_n\right)-\left(f^{\prime}-\mu_n\right)$. Therefore, in order to show that $\mathrm{Ker}(\varepsilon)$ is contained in $\mathrm{Im}(\rho)$ it suffices to show that $f-\mu_n$ is in the image of $\rho$ for each $f\in \mathrm{Hom}_{\mathcal{I}\mathcal{F}(as)}\left([n],[0]\right)$.

Consider a morphism $g\in \mathrm{Hom}_{\mathcal{I}\mathcal{F}(as)}\left([n],[2]\right)$. This map is determined by $g^{-1}(0)=X$, $g^{-1}(1)=Y$ and $g^{-1}(2)=Z$ where $X$, $Y$ and $Z$ are $C_2$-labelled, ordered sets. By definition $\rho(g)=\mu_2\circ g - \nu\circ g$ where $\left(\mu_2\circ g\right)^{-1}(0)=X<Y<Z$ and $\left(\nu\circ g\right)^{-1}(0)=Z<Y^t<X$.

In order to show that the differences $f-\mu_n$ lie in the image of $\rho$ it suffices to show that there is a finite sequence of morphisms $f_p,\dotsc ,f_0$ in $\mathrm{Hom}_{\mathcal{I}\mathcal{F}(as)}\left([n],[0]\right)$ such that $f_p=f$, $f_0=\mu_n$ and for each pair $(f_i,f_{i-1})$ there exists a $g\in \mathrm{Hom}_{\mathcal{I}\mathcal{F}(as)}\left([n],[2]\right)$ such that $\mu_2\circ g=f_i$ and $\nu\circ g=f_{i-1}$.

We construct this sequence of morphisms as follows.
\begin{enumerate}
\item Let $f\in \mathrm{Hom}_{\mathcal{I}\mathcal{F}(as)}\left([n],[0]\right)$ be defined by 
\[f^{-1}(0)=\llb i_0^{z_0}<\cdots <i_n^{z_n}\rrb.\] 
\item If $z_0=z_1=\cdots = z_n=1$ skip to Step \ref{cyc1}. Otherwise apply Steps \ref{inv1} to \ref{inv3} until this is the case.
\item \label{inv1} If $z_0=t$ we choose $g$ with 
\[X=\emptyset , \quad Y=\llb i_0^t\rrb, \quad Z=\llb i_1^{z_1}<\cdots <i_n^{z_n}\rrb.\]
\item \label{inv2} If $z_0=z_1=\cdots = z_{j-1}=1$ with $j<n$, we choose $g$ with \[X=\llb i_0^1<\cdots i_{j-1}^1\rrb , \quad Y=\llb i_j^t\rrb , \quad Z= \llb i_{j+1}^{z_{j+1}}<\cdots < i_n^{z_n}\rrb.\]
\item \label{inv3} If $z_0=z_1=\cdots z_{n-1}=1$ we choose $g$ with 
\[X=\llb i_0^1<\cdots i_{n-1}^1\rrb , \quad Y=\llb i_n^t\rrb , \quad Z=\emptyset.\]
\item \label{cyc1} By applying the steps above we have obtained a morphism, say $\beta$, in $\mathrm{Hom}_{\mathcal{I}\mathcal{F}(as)}\left([n],[0]\right)$ such that $\beta^{-1}(0)$ is of the form
\[\llb k_0^1<k_1^1<\cdots <k_n^1\rrb.\]
\item If $k_m=m$ for $0\leqslant m\leqslant n$ then we have obtained $\mu_n$ and we are done.
\item If not, by repeatedly choosing $g$ such that $Y=\emptyset$ and $Z$ is a singleton we cyclically permute the elements of the total ordering without changing any of the labels. We do this until $k_n=n$.
\item If $k_m=m$ for $0\leqslant m\leqslant n$ then we have obtained $\mu_n$ and we are done.
\item \label{hyp1} If not, suppose $k_m=m$ for $m>j$. Choose $g$ with 
\[X=\llb k_0^1<\cdots k_{j-1}^1\rrb , \quad Y=\llb k_j^1\rrb, \quad Z=\llb k_{j+1}^1<\cdots <k_n^1\rrb.\]
\item By choosing $g$ such that $Y=\emptyset$ and $Z$ is singleton, we can cyclically permute this total ordering until we obtain
\[\llb k_j^t<k_0^1<\cdots k_{j-1}^1<k_{j+1}^1<\cdots k_n^n\rrb.\]
\item Choose $g$ with $X=\emptyset$, $Y=\llb j_k^t\rrb $ and $Z=\llb  k_0^1<\cdots k_{j-1}^1<k_{j+1}^1<\cdots k_n^n\rrb$.
\item \label{hyp2} By choosing $g$ with $Y=\emptyset$ and $Z$ a singleton, we cyclically permute the resulting total ordering to get
\[\llb k_j^1< k_0^1< \cdots < k_{j-1}^1<k_{j+1}^1<\cdots k_n^n\rrb.\]
\item The result of Steps \ref{hyp1} to \ref{hyp2} is to fix $k_m$ for $m>j$ and to cyclically permute $k_0$ up to $k_j$ whilst leaving the labels unchanged. Repeating this process a finite number of times we will obtain $\mu_n$ as required. 
\end{enumerate}
This shows that all the differences of the form $f-\mu_n$ lie in the image of $\rho$ and so $\mathrm{Ker}(\varepsilon)\subseteq \mathrm{Im}(\rho)$. Therefore the sequence is exact.
\end{proof} 

Let $F\colon \mathcal{I}\mathcal{F}(as)\rightarrow \Delta H$ denote the isomorphism of \cite[Theorem 1.7]{GravesHHA}. The map $F\left(\mu_n\right)$ is the unique order-preserving map $\left(\mu_n, id_n\right)\in \mathrm{Hom}_{\Delta H}\left([n],[0]\right)$. Furthermore, the map $F\left(\nu\right)$ is the map $\left(\mu_2 , \left(1,t,1;(0\,2)\right)\right) \in \mathrm{Hom}_{\Delta H}\left([2],[0]\right)$.

\begin{cor}
\label{DHCor}
For each $n\geqslant 0 $ there is an exact sequence of $k$-modules
\[0\leftarrow k\xleftarrow{\varepsilon} k\left[\mathrm{Hom}_{\Delta H}\left([n],[0]\right)\right] \xleftarrow{\rho} k\left[\mathrm{Hom}_{\Delta H}\left([n],[2]\right)\right]\]
where $\varepsilon$ is defined by $\varepsilon\left( \left(\varphi , g\right)\right)=1_k$ on generators and extended $k$-linearly. The map $\rho$ is determined on generators by $\rho\left(\left(\psi , h\right)\right)=F\left(\mu_2\right)\circ \left(\psi , h\right) - F\left(\nu\right) \circ \left(\psi , h\right)$ and extended $k$-linearly.
\end{cor}
\begin{proof}
This follows directly from Lemma \ref{ESLem} and \cite[Theorem 1.7]{GravesHHA}.
\end{proof}

\begin{cor}
\label{PRCor}
There is an exact sequence of right $\Delta H$-modules 
\[0\leftarrow k^{\star} \xleftarrow{\varepsilon}  k\left[\mathrm{Hom}_{\Delta H}\left(-,[0]\right)\right] \xleftarrow{\rho} k\left[\mathrm{Hom}_{\Delta H}\left(-,[2]\right)\right]\]
where $\varepsilon$ and $\rho$ are the natural transformations formed from the $k$-module maps of Corollary \ref{DHCor}.
\end{cor}
\begin{proof}
The exactness of a sequence of functors is checked object-wise, so this follows directly from Corollary \ref{DHCor}.
\end{proof}

\begin{cor}
\label{HO-Cor-1}
Let $A$ be a unital, involutive $k$-algebra. The homology of the partial chain complex
\[0\leftarrow A \xleftarrow{d} A^{\otimes 3}\]
where $d(a_0\otimes a_1 \otimes a_2) = a_0a_1a_2-a_2\ol{a_1}a_0$ is isomorphic to $HO_{0}(A)$.
\end{cor}
\begin{proof}
We apply the functor $-\otimes_{\Delta H} \mathsf{H}_A$ to the partial resolution of $k^{\star}$ in Corollary \ref{PRCor} and take homology.
\end{proof}

\begin{lem}
\label{ideal-lem}
The $k$-submodule $k\left[\llb a_0a_1a_2-a_2\ol{a_1}a_0\rrb\right]$ of $A$ is, in fact, an ideal of $A$.
\end{lem}
\begin{proof}
We observe that given a product of four elements $a_0a_1a_2a_3$ we can use the relation to obtain any product of four elements given by applying the involution to any subset of the four factors, then permuting the four factors. This follows from the algorithm of Lemma \ref{ESLem} by showing that we can use the relation to obtain $\mu_3$ from any map in $\mathrm{Hom}_{\mathcal{I}\mathcal{F}(as)}\left([3],[0]\right)$.
\end{proof}

\begin{thm}
\label{HOTHM}
Let $A$ be a unital, involutive $k$-algebra. There is an isomorphism
\[HO_0(A)\cong \frac{A}{(a_0a_1a_2-a_2\ol{a_1}a_0)}.\]
In particular, $HO_0(A)$ has a ring structure induced from the ring structure of $A$.
\end{thm}
\begin{proof}
This follows from Corollary \ref{HO-Cor-1} and Lemma \ref{ideal-lem}.
\end{proof}

\begin{rem}
We noted in the introduction that hyperoctahedral homology has the remarkable property that it has a graded-commutative product structure even when the algebra under consideration is not commutative. We see this in action with this theorem. By taking $a_1$ to be the multiplicative identity $1_A$, we recover the ideal generated by the commutator submodule of $A$. This appears in degree zero symmetric homology \cite[Theorem 86]{Ault}.
\end{rem}

\begin{cor}
\label{zero-cor}
Let $A$ be a unital, involutive $k$-algebra.
\begin{enumerate} 
\item \label{comm-triv} If $A$ is a commutative algebra with the trivial involution then $HO_0(A)\cong A$.
\item Let $A$ be a commutative algebra with a non-trivial involution. In this case $HO_0(A)$ is isomorphic to the coinvariants of $A$ under the involution.
\item \label{triv-cor} Let $A$ be a unital involutive $k$-algebra. If the ideal $\left(a_0a_1a_2-a_2\ol{a_1}a_0\right)$ is equal to $A$, then $HO_{\star}(A)$ is trivial in all degrees.
\end{enumerate}
\end{cor}
\begin{proof}
The first two statements follow directly from Theorem \ref{HOTHM}. The third statement follows from Theorem \ref{HOTHM} and Corollary \ref{Main-cor}. In particular, since we have a unital, associative, graded-commutative product we have have surjective maps $HO_0(A)\otimes HO_q(A)\rightarrow HO_q(A)$ given by $1\otimes x \mapsto x$ for all $q\geqslant 0$. However, if the ideal $\left(a_0a_1a_2-a_2\ol{a_1}a_0\right)$ is equal to $A$, $HO_0(A)$ is trivial and so $HO_{\star}(A)$ is trivial in all degrees.
\end{proof}

\begin{cor}
Hyperoctahedral homology fails to preserve Morita equivalence.
\end{cor}
\begin{proof}
Let $A$ be a commutative algebra with trivial involution. Consider the involutive algebra of $(n\times n)$-matrices, $M_n(A)$, with $n>1$. By taking $a_1$ to be the identity matrix we see that the ideal $\left(a_0a_1a_2-a_2\ol{a_1}a_0\right)=M_n(A)$. By Statement (\ref{triv-cor}) of Corollary \ref{zero-cor}, the hyperoctahedral homology of $M_n(A)$ is trivial in all degrees. However, by Statement (\ref{comm-triv}) of Corollary \ref{zero-cor} $HO_0(A)\cong A$. Therefore, hyperoctahedral homology does not preserve Morita equivalence.
\end{proof}

\begin{prop}
When restricted to $HO_0(A)\otimes HO_0(A)\rightarrow HO_0(A)$, the Pontryagin product of Corollary \ref{Main-cor} is the algebra multiplication map in 
\[\frac{A}{\left(a_0a_1a_2-a_2\ol{a_1}a_0\right)}.\]
\end{prop}
\begin{proof}
Using the degree zero and degree one terms of the chain complex $C_{\star}\left( \Delta H_{+}, \mathsf{H}_{A+}\right)$ one can show that $HO_0(A)$ is generated $k$-linearly by equivalence classes of the form $id_{[0]}\otimes a$. One then observes that the Pontryagin product structure of Corollary \ref{Main-cor} is the algebra multiplication as required.
\end{proof}

\bibliographystyle{alpha}
\bibliography{HOHomOps-refs}

\end{document}